\documentclass[11pt,a4paper]{amsart}
\usepackage{graphicx}

\usepackage{amsmath}%
\usepackage{amssymb,indentfirst,calc,mathrsfs,euscript,bbm}%

\usepackage{setspace}%
\usepackage[reals]{layout}%
\usepackage{xr}%
\usepackage{amscd}%
\usepackage{comment}
\usepackage{epstopdf}

\def\rit{\mathbb R}

\def\val{\hbox{\tt val}}

 \def\P{{\mathsf P}} 
 \def\E{{ \mathsf E}} 

\def\val{\hbox{\tt val}}

\def\wl{\hbox{$\lim_{{\lambda}{\rightarrow}0}
v_{\lambda}$}}

\def\D{{\bf D}}
\def\T{{\bf T}}

\def\wl{w_{\lambda}}
\def\Wl{W_{\lambda}}

\newtheorem{lem}{Lemma}[section]
\newtheorem{defi}{Definition}[section]

\newtheorem{rem}{Remark}[section]

\newtheorem{cor}{Corollary}[section]
\newtheorem{pro}{Proposition}[section]

\begin{document}

\title{Operator approach to values of  stochastic games with varying  stage duration}

\author{
Sylvain Sorin}
\address{
\textbf{Sylvain Sorin}\\
Sorbonne Universit\'es, UPMC Univ Paris 06, Institut de Math\'ematiques de Jussieu-Paris Rive Gauche, UMR 7586, CNRS,
 Univ Paris Diderot, Sorbonne Paris Cit\'e, F-75005, Paris, France }
\email{
sylvain.sorin@imj-prg.fr\newline
http://webusers.imj-prg.fr/~sylvain.sorin/
}
\thanks{This was co-funded by 
PGMO  2014-LMG. The second author was partially supported by the French Agence Nationale de la Recherche (ANR) "ANR GAGA: ANR-13-JS01-0004-01". }

\author{Guillaume Vigeral }
\address{\textbf{Guillaume Vigeral}\\
Universit\'e Paris-Dauphine, CEREMADE, Place du Mar{\'e}chal De Lattre de Tassigny. 75775 Paris cedex 16, France}
\email{
vigeral@ceremade.dauphine.fr\newline
http://www.ceremade.dauphine.fr/~vigeral/indexenglish.html
}
\date{ September   2015}


\bibliographystyle{apalike}

\begin{abstract}
We study the links between the values of  stochastic  games   with varying stage duration $h$,  the corresponding  Shapley operators $\T$ and $\T _h= h\T + (1-h ) Id$
and the solution of the evolution equation $\dot f_t = (\T - Id )f_t$. Considering general non expansive maps  we establish two kinds of results,  under both  the discounted or the finite length framework,  that apply to the class of ``exact" stochastic games. First, for a fixed length or discount factor, the  value converges as the stage duration go to 0. Second, the asymptotic behavior of the  value as the length goes to infinity, or as the discount factor goes to 0, does not depend on the stage duration. 
In addition, these properties imply the existence of the value of the finite length or discounted continuous time game  (associated to a continuous time jointly controlled Markov process), as the limit of the value of any  discretization with vanishing mesh.\end{abstract}

\maketitle

\section{Introduction}
The operator introduced by Shapley  \cite{Sh} to study zero-sum discounted stochastic games is a non expansive map $ \T$ from a Banach space to itself. Several results have been obtained by using a similar  ``operator approach"   in the framework of zero-sum repeated games,   \cite{RS},  \cite{ NNE}, \cite{S03}, \cite{NS10}.
In particular the analysis extends to general repeated games (including incomplete information and signals, see  \cite{MSZ} Chapter IV) and arbitrary evaluation of the sequence of stage payoffs. 

An important part of the literature studies families of evaluations with vanishing stage weight (either length going to infinity or discount factor going to 0) and the main issue is the existence of an asymptotic value. Assuming that the stage duration is one, each evaluation induces a time ponderation on $\rit^+$ and vanishing stage weight leads to an increasing number $n_a$ of interactions during any given fraction $a\in]0,1[$ of the game that has been played according to this ponderation.

We consider here another direction of research: the time ponderation is fixed and the stage duration vanishes (leading to a continuous time game at the limit). Note that, as above, this leads to an increasing number $n_a$ of interactions. 

We study in particular stochastic games with varying stage duration, in the spirit of   Neyman  \cite{N2013}. Our approach is based on the non expansive property of the Shapley operator to derive convergence results, characterization of the values, and links with evolution equations in continuous time.
\\

\noindent The structure of the paper is as follows:\\
We first recall the  definition of  the Shapley operator $T$ associated to a stochastic game (Section  2) and the related finite and discounted  iterations.\\
 We introduce in Section 3 two models of stochastic games with  variable stage duration $h$  : linearization via ``exact" games, and ``discretization" of a continuous time model. In both frameworks we describe  the link with the fractional Shapley operator $\T_h$.\\
 Sections 4 and 5 are  devoted to the abstract analysis of various fractional  iterations of a  general  non expansive map $\T$:\\
- first in the finite iteration case, where we establish relations between  the  $n$-iterate $ \T_h ^n$   and  the solution of the evolution equation $\dot f_t = (\T - Id )f_t$ at time $t = nh$,\\
- then in the discounted case,  where we identify the $\lambda$-discounted evaluation associated to $\T_h$ as the $ \mu = \frac {\lambda}{1+\lambda - \lambda h} $-discounted evaluation associated to $\T$.\\
We then  apply these results to the case of exact stochastic games. Section 6 and 7 are respectively devoted to the study of games with finite length and with  discount factor. In both frameworks we establish results of two different kinds. Firstly, for a fixed evaluation (finite length or discount factor), the value of a game with varying stage duration converges as the stage duration goes to 0. 
  Secondly, the asymptotic behavior of the  value (for large length or small discount factor) does not depend on the stage duration. \\
   In Section 8 we study the discretization of the continuous time game by approximating with exact games and we prove convergence of the values in the finite length and discounted case as the stage duration vanishes.\\
    The last section provide concluding comments.
   
\section{Stochastic games and Shapley operator }
Consider a  two person zero-sum  stochastic  game $G$ with a finite state space $\Omega$. $I$ and 
$J$  are compact  metric action spaces, $X$ and $Y$ are the sets of regular probabilities on the corresponding Borel $\sigma$-algebra. $g$  is  a bounded measurable payoff function from $ \Omega \times I \times J $ to $\rit $  (with multilinear extension to $X\times Y$)  and for each $(i,j) \in  I \times J, $  $P(i,j)$  is a   transition probability from $ \Omega $ to $\Delta (\Omega)$ (the set of probabilities on $\Omega$). $g$ and $P$ are separately continuous on $I$ and $J$.

The game is played in stages. At stage $n$, knowing the state $\omega_n$, player 1 (resp. 2) chooses $i_n \in I$ (resp. $j_n \in J$), the stage payoff is $g_n = g( \omega_n, i_n, j_n)$. The next state $ \omega_{n+1}$  is selected according to the probability $ P(i_n, j_n)[\omega_n]$ and is announced  to the players.\\

One associates  to $G$ a  Shapley operator, see Shapley   \cite{Sh}, which  is a map $\T$  from $F = \rit ^\Omega$  to itself: $ f \in F \mapsto \T (f)$  defined  by 
\begin{equation}\label{gT}
\T (f) (\omega) = \underset{(x,y) \in X\times Y}{\val} \{ g(\omega; x, y)  + P( x, y) [\omega] \circ  f \}, \qquad \forall \omega \in \Omega
\end{equation}
where $ \underset{X\times Y}{\val} $ is the $\max\min= \min\max = $ value operator on $X\times Y$, \\
$P( x, y) [\omega] (\omega ') = \int_{I\times J}  P( i, j) [\omega] (\omega ') x(di) y(dj) $ and  for $R \in \rit^\Omega$, $R\circ f = \sum_{\zeta \in \Omega} R(\zeta) f(\zeta)$.\\
Note that $\T$ is a non expansive map. Moreover $\T$ is monotone and translates the constants (for a converse result see, e.g., Kolokoltsov \cite{K92},  and Sorin \cite{S04} for related consequences) but we will not use here  these additional properties. \\

One can consider two other frameworks with $\Omega$ standard Borel,    where 
$\T$ is defined in a similar way  with $P( x, y) [\omega] \circ  f = \int _{\Omega} f(\zeta)  P( x, y) [\omega] (d\zeta)$  and  where $F$ is  either:\\
-  the set of  bounded measurable  functions on $\Omega$  and $P(i,j) [\omega] (A) $ is separately continous in  $(i,j)$ for each  Borel subset $A \subset \Omega$ (see \cite{MSZ},  Prop. VII.1.4),\\
- or  the set  of bounded  continuous functions on $\Omega$ and both maps $(x, \omega) \mapsto  I (x,y; \omega) =   \int _{\Omega} f(\zeta)  P( x, y) [\omega] (d\zeta)$ and  $(y, \omega) \mapsto  I (x,y; \omega)$ are continuous  for any bounded continuous function $f$ on $\Omega$  (see \cite{MSZ},  Prop. VII.1.5).\\
For more general conditions see Nowak  \cite{No85}, \cite{No03}.

From now on  we  assume that one of these cases  holds so that $\T$ is well defined from some Banach space  $F$ to itself.\\

Recall that  $V_n = \T ^n (0)$ is the value of the  $n$-stage game with total evaluation  $\sum_{m = 1}^n  g_m$ (as a function of the initial state) so that the normalized value is $v_n = \frac{V_n}{n}$.\\ 
$W_{\lambda} $,  which is the unique fixed point of $f \mapsto \T ((1- \lambda) f)$ on $F$,   is the un-normalized value of the discounted game with total evaluation  $\sum_{m = 1}^{\infty}  g_m (1- \lambda)^{m-1}$ and the normalized discounted value is $w_\lambda= \lambda W_{\lambda} $.


\section{Stochastic games with varying stage duration}

Let  us introduce, for each $(i,j)\in I\times J$, the kernel $Q(i,j)$ such that  $P( i, j) = Id + Q( i,j)$
and  write  $G= (g, Q)$ for  a stochastic game  defined as above.
One introduces two  families  of varying  stage duration games, see Neyman  \cite{N2013}, associated to $G$.

 \subsection{Exact sequence}$ $ \\
 Consider $G$ as a game with stage duration one. Given a step size $h \in (0, 1]$,
 define an  ``exact"  game  $ G^h$ with  stage duration $h$, stage payoff $hg$ and stage  transition $P_h = Id + h\: Q$. That is, $G^h = (h\:g, h\:Q)$.\\
 $G^h$ appears as a linearization of the game $G$. During a stage of duration $h$ both the payoff and the state variation  are proportional with factor $h$ to those of a stage of duration one.
 \begin{defi}
Given $ h \in [0,1]$, let 
$
 \T_{h} =  (1-h)  Id + h\: \T.
 $
\end{defi}

 Then one has:
 \begin{pro}$ $\\
If $\T$ is the Shapley operator of $G$, then $\T_h$ is the Shapley operator of the game $G^h$.
\end{pro}
\begin{proof}
Since  $  \T_{h} (f) = (1-h) f + h \: \val \{  g +  P \circ  f \}= (1-h) f + \val \{ h\: g +   h (Id +  Q)  \circ  f \}$, one obtains 
\begin{equation}\label{gTmu}
 \T_{h} (f) = \val \{ h\: g +  P_{h}\circ  f \}
\end{equation}
 with
 $P_{h} = Id + h Q$. \\
Hence $  \T_{h} $ is the one stage operator associated to the  game $G^h$,.
\end{proof}
We will consider the associated finitely repeated games and discounted games asociated to $G^h$.
 Natural questions are, in the finite case : \\
1) given a total length $M$, what is the asymptotic behavior of the value of the $N$-stage game with stage duration $h$, as $h$ vanishes and $Nh= M$. \\
2) what is the asymptotic behavior  of the value, as $Nh$ goes to $\infty$,\\
and similarly in the discounted framework.\\

These topics will be addressed in the  general setting of a non expansive map $\T$  in Sections 4 and 5.
In both cases we will obtain explicit formulations for the limits.

 \subsection{Discretization}$ $ \\
 Let  $G= (g, Q)$ be a stochastic game with a {\it finite} state space. 
We consider here a continuous time jointly controlled Markov  process associated to  the kernel $Q$.\\
Explicitly, define $\P ^t (i, j)$ as the continuous time homogeneous Markov chain  on $\Omega$, indexed by $\rit ^+$, with generator  $Q(i, j)$:
\begin{equation}\label{evol}
 \dot \P ^t(i,j)  = \P ^t(i,j) Q(i,j).
\end{equation}
Given a stepsize $h \in (0, 1]$,
 $\overline G ^h$ has to be considered as the discretization with mesh $h$ of the game in continuous time  $\overline G$ where the state variable follows  $\P^t$ and is  controlled by  both players, see  \cite{Za64}, \cite{TW}, \cite{GH1},  \cite{N2012}.\\
More precisely  the players act at  time $s= kh$ by choosing   actions $(i_s, j_s)$  (at random according to  some ${x_s}$, resp. ${y_s}$), knowing the current    state.
 Between time $s$ and $s+h$,  the state  $\omega_t$ evolves  with  conditional law $\P^t$ following (\ref{evol}) with $Q(i_s, j_s)$  and   $\P^s = Id$.

 The associated Shapley operator  of this stochastic game is $ \overline \T _h $ with
$$
 \overline \T _h (f)    = \underset{X\times Y}{\val} \{  g^h +  \P^h \circ    f \}
 $$
 where $g^h (\omega_0, x,y)$ stands for $\E [ \int _0^h g( \omega_t;  x, y )dt]$ and $\P^h ( x, y) = \int_{I\times J} \P^h (i, j) x(di) y(dj)$.

The corresponding finitely repeated and discounted games will be analyzed in Section 8.

\section{Finite iterations of non expansive maps and evolution equations}

Consider a non expansive map $\T$ from a Banach space $Z$ to itself. In this section we  recall basic results concerning its iterations and the corresponding discrete and continuous dynamics.
\subsection{Finite iteration}$ $\\
The   $n$-stage iteration starting from $z\in Z $ is   $U_n = {\T}^n (z)$ hence satisfies 
$$
U_n - U_{n-1} =  - (Id - {\T})(U_{n-1})
$$
 which can be considered as a discretization of the differential equation
\begin{equation}\label{M}
\dot f_t = -(Id - \T) f_t,\quad f_0 = z.
\end{equation}
(Note that this is a special case of the differential inclusion $  \dot f_t  \in  - A f_t$, for the accretive (maximal monotone) operator  $A = Id - {\T}$.)\\

The comparison between the iterates of $\T$ and the solution  $f_t(z)$ of the differential equation (\ref{M})  is given by the generalized Chernoff's formula \cite{MO}, \cite{BP},  see, e.g., Br\'ezis  \cite{Brezis}, p.16:
\begin{pro}
\begin{equation}\label{T}
\| f_t(z)  - {\T}^n (z))\| \leq \| z - \T(z) \| \sqrt { t + (n-t)^2}.
\end{equation}
\end{pro}$ $ \\
In particular with $z = 0$ and  $t=n$, one obtains
\begin{equation}\label{Vn}
\| \frac{f_n(0)}{n}  -  v_n \| \leq \frac{\| {\T} (0) \| } {\sqrt { n }}
\end{equation}
where as before, $\T ^n (0) =V_n=  n v _n $.\\
Given  $h \in (0,1]$, a   change of time shows that   $ f _{t/h}(z)$ is the  solution of
\begin{equation}\label{M'}
\dot g_t =- \frac{ (Id - {\T}) g_t} {h},\quad g_0 = z.
\end{equation}

\subsection{Interpolation}$ $\\
Given  $h \in [0,1]$ introduce again:
\begin{equation}\label{TMu}
\T _{h} = (1-h)  Id + h\: \T .
\end{equation}
 Then using (\ref{M'})  which is 
$$
\dot g_t =- { (Id - {\T}_h) g_t} ,\quad g_0 = z,
$$
one obtains from \eqref{T}
\begin{equation}\label{Tmu}
\| f_t(z)  - \T_{h}^n (z)\| \leq \| z - \T z\| \sqrt { {t}{h} + (n h - {t})^2},
\end{equation}
hence in particular  with $h = \frac{t}{n}$
\begin{equation}\label{itert}
\| f_t(z)  -  \T_{t/n}^n (z)\| \leq \| z - \T z\| \frac{t}{ \sqrt n}  \, ,
\end{equation}
or
\begin{equation}\label{itermu}
\| f_{n h} (z)  -  \T_{h}^n (z)\| \leq \| z - \T z\|\,  h \sqrt{n}.
\end{equation}


\subsection{Eulerian schemes} $ $ \\
More generally for a  sequence  of step sizes 
$ \{h_{k}\}$ in $[0,1]$ one defines inductively an Eulerian scheme $\{z_k\}$ by
$$
z_{k+1}-z_k  = h_{k+1}(T-Id)(z_k) 
$$
or 
$$
z_{k+1}=\T_{h_{k+1}} z_{k}.
$$
For two  sequences  
$\{ h_{k}\}, \{ \hat h_{\ell}\}$ in $[0,1]$, with  associated Eulerian schemes
$$
z_{k+1}  = \T_{h_{k+1}} z_{k},
$$
$$
\hat z_{\ell+1}  = \T_{\hat h_{\ell+1}}  \hat z_{\ell},
$$
Vigeral \cite{V1} obtains 

\begin{pro}\label{kobavig}
\begin{equation}\label{v1}
\| \hat z_{\ell}   - z_k  \| \leq \| \hat z_0 - z \| + \| z_0 - z \| +  \| z - \T z\|\,   \sqrt{ (\sigma_k - \hat \sigma_{\ell} ) ^2 + \tau_k + \hat  \tau_{\ell}},\quad  \forall z\in z,
\end{equation}
\begin{equation}\label{v2}
\| f_{t} (z)  - z_k  \| \leq \| z - \T z\|\,   \sqrt{ (\sigma_k - t ) ^2 + \tau_k },
\end{equation}
with $z_0 = z$, $\sigma_k = \sum_{i = 1}^k h _ i$ , $ \tau _k = \sum_{i = 1}^k h _ i ^2$, $\hat \sigma_\ell = \sum_{j = 1}^\ell \hat h _ j$ , $\hat \tau _\ell = \sum_{j = 1}^\ell \hat h _ j ^2$.
\end{pro}$ $ \\
In particular this gives, in the uniform case $h_i = h, \forall i$
\begin{equation}
\| f_t(z)  - \T_{h}^n (z))\| \leq \| z - \T z\| \sqrt { n h ^2+ (n h - {t})^2}
\end{equation}
 and coincides with (\ref{itert}) and (\ref{itermu}) at $t = n h$.

\subsection{Two approximations}$ $ \\
Equation (\ref{Tmu}) or more generally  (\ref{v2}) corresponds to two  approximations:

i) Comparison on a compact interval  $[0,M]$ of $f_t$   to   the linear interpolation $ \hat  \T_{M/n}^{s}$ of $\{\T_{M/n}^m\},  m= 0, ..., n $,  which is,   using (\ref{itert}) :
$$
\| f_t(0)  - \hat  \T_{M/n}^{nt/M} (0)\| \leq  K \frac{M}{ \sqrt n}, \ \forall t \in [0, M ],
$$
 for some constant $K$. \\
 Or more generally if one considers a sequence of step sizes  $\{h_i\}, i= 1, ..., k$ with $\sigma_k = \sum_{i=1}^k  h_i = M$, $h^i \leq h, \forall i$ and $ \Pi _i\T_{h_i} = \T_{h_1}\circ \cdots \circ \T_{h_k}$, one has:
\begin{equation}\label{various1}
\| f_M(0)  - \Pi _i\T_{h_i} (0)\| \leq  K  \sqrt {h M}.
\end{equation}$ $\\
Thus the composite iteration $ \Pi _i\T_{h_i}$ converges to the solution of (\ref{M}) as the mesh $h$ goes to 0.

ii) Asymptotic comparison of the  behavior of $f_t$, solution of (\ref{M}) and iterations  of the form  $\Pi _i\T_{h_i}$ with step size $h_i \leq 1$ and  total length $ \sigma_k = t$:
\begin{equation}\label{various2}
\| f_t (0)  -\Pi _i\T_{h_i} (0) \|\leq K \sqrt t.
\end{equation}

\section{Discounted iterations of non expansive maps}
\subsection{General properties}$ $ \\
For $\lambda \in (0,1]$ denote again  by  $ \Wl$  the unique fixed point  of $z \mapsto T ((1- \lambda) z)$ and  let $\wl= \lambda \Wl$.\\
We recall some basic evaluations, see e.g. \cite{V0}.

\begin{pro}\label{basiclem}
\begin{eqnarray*}
\|\wl \|&\leq& \| \T (0)\|,\\
\| \wl - w_\mu \| &\leq &2| 1 - \frac{\lambda}{\mu} |  \|\T (0)\|.
\end{eqnarray*}
\end{pro}
\begin{proof}
First, one has:
 $$
\|\Wl \| - \|\T (0) \| \leq \|\Wl  - \T (0) \| = \|\T ( [1- \lambda] \Wl )  - \T (0) \| \leq  (1- \lambda) \|\Wl \|
$$
which implies $ \lambda \|\Wl \|  \leq  \| \T (0) \|$.\\
Moreover :
\begin{eqnarray*}
\|\Wl - W_\mu \|& =&  \|\T ( [1- \lambda] \Wl) - \T ( [1- \mu]  W_\mu) \| \\
&\leq& \| [1- \lambda] \Wl -  [1- \mu]  W_\mu \|\\
&\leq& (1- \lambda) \|\Wl - W_\mu \| +  | \lambda - \mu \| \| W_\mu \|
\end{eqnarray*} thus
$$
 \lambda \|\Wl - W_\mu \| \leq   | \lambda - \mu \| \| W_\mu \|\leq \frac{| \lambda - \mu |}{\mu} \| \T(0) \|.
 $$
On the other hand:
$$
\| \wl - w_\mu \| \leq  \lambda \|\Wl - W_\mu \| +  | \lambda - \mu \| \| W_\mu \|
$$
hence
$$
\| \wl - w_\mu \| \leq 2   | \lambda - \mu \| \| W_\mu \|
$$
and the result  follows from $ \| W_\mu \| \leq \| \T (0) \| / \mu$.
\end{proof}

\subsection{Discounted values}$ $ \\
For any non expansive operator $\T$ on $Z$ and $h \in (0,1]$, introduce 
\[
W_\lambda^h=\T_h((1-\lambda h)W_\lambda^h)
\]
as the unique fixed point point of $u\mapsto \T_h ( (1-\lambda h)u)$ and define 
\begin{equation}\label{dh}
w_\lambda^h=\lambda W_\lambda^h = \lambda \T_h(\frac{1-\lambda h}{\lambda}w_\lambda^h).
\end{equation}
$W_\lambda^h$ is the un-normalized $\lambda$-evaluation computed through  a stage of duration $h$   using  the linearization $\T_h$ of $ \T$ and $w_\lambda^h$ is  the associated normalization. \\
Recall that for $h=1$, $\T_h = \T$ and $w_\lambda= w_\lambda^h$.

\begin{pro}\label{dcv}
$$
w_{\lambda}^h=w_\mu, \: \: \mbox{with }\ \mu= \frac{\lambda}{1+ \lambda - \lambda \: h }.
$$
\end{pro}

\begin{proof}
By definition of $\T_h$,
\begin{eqnarray*}
w_{\lambda}^h&=&\lambda ((1-h)Id+h\T)(\frac{1-\lambda h}{\lambda}w_\lambda^h)\\
&=&(1-h)(1-\lambda h)w_{\lambda}^h+\lambda h \T(\frac{1-\lambda h}{\lambda}w_\lambda^h).
\end{eqnarray*}
Hence 
\[
(1+\lambda-\lambda h)w_{\lambda}^h=\lambda \T(\frac{1-\lambda h}{\lambda}w_\lambda^h)
\]
which is   $$
w_{\lambda}^h=\mu \T(\frac{1-\mu}{\mu}w_\lambda^h)
$$
 for $\mu= \frac{\lambda}{1+ \lambda - \lambda \: h }$. \\
 The non expansiveness of $\T$ yields uniqueness, hence  the result.
\end{proof}

\subsection{Vanishing duration}$ $ \\
Introduce $ \D^h_\lambda$,   the auxiliary one stage operator associated to the $\lambda$-discounted  evaluation of $\T_h$, defined by
$$
 \D^h_\lambda \:z = \lambda \T_h[(\frac{1-\lambda h}{\lambda}) z]
 $$
 which is $(1 - \lambda h) $ contracting. \\
 In particular $ \D^h_\lambda\:  w_{\lambda}^h = w_{\lambda}^h$ and for $h=1$,  $ \D^1_\lambda\:  w_{\lambda} = w_{\lambda}$.
 \begin{pro}\label{d1}
 For any $z \in Z$
$$
w_{\lambda}^h=  \lim_{n \rightarrow + \infty} (\D^h_\lambda)^n z
$$
and $ w_{\lambda}^h   \rightarrow   w_\frac{\lambda}{1+ \lambda }$ as $h \rightarrow 0$ with
$$
\|w_{\lambda}^h- w_{\frac{\lambda}{1+\lambda}}\|\leq C \lambda h.
$$
\end{pro}
\begin{proof}
The first equality follows from definition  (\ref{dh}).\\
By Proposition \ref{basiclem}
\begin{eqnarray*}
\left\|w_\lambda^h-w_{\frac{\lambda}{1+\lambda}}\right\|&\leq&C\|1-\frac{1+\lambda-\lambda h}{1+\lambda}\|\\ 
&\leq&C\lambda h.
\end{eqnarray*}
\end{proof}
More generally one can consider a sequence of step sizes $\{h_i\}$ with $h_i \leq h$ and $\sum_i h_i = + \infty$ and the associated operator $ \Pi_i  \D^{h_i} _\lambda$.

 \begin{lem}
  For any $z \in Z$ and any sequence $h_1,\cdots,h_n$, 
  \[
   \|\Pi_{i=1}^n  \D^{h_i} _\lambda (z)-w_{\frac{\lambda}{1+\lambda}} \|\leq 2\|\T(0)\| \max_{1\leq i\leq n} h_i +   (\|\T(0)\|+\|z\|) \Pi_{i=1}^n (1-\lambda h_i).
    \] 
\end{lem}

\begin{proof}
By non expansiveness, \[
 \|\Pi_{i=1}^n  \D^{h_i} _\lambda (z)  - \Pi_{i=1}^n  \D^{h_i} _\lambda (w_{\frac{\lambda}{1+\lambda}})\| \leq  \|z-w_{\frac{\lambda}{1+\lambda}}\|\:  \Pi_{i=1}^n (1-\lambda h_i)\leq   (\|\T(0)\|+\|z\|) \: \Pi_{i=1}^n (1-\lambda h_i).
 \]
Hence it is enough to show that  $\|\Pi_{i=1}^n  \D^{h_i} _\lambda (w_{\frac{\lambda}{1+\lambda}})-w_{\frac{\lambda}{1+\lambda}} \|\leq 2\|\T(0)\| \max_{1\leq i\leq n} h_i$.\\
Let $d_k=\|\Pi_{i=k}^n  \D^{h_i} _\lambda (w_{\frac{\lambda}{1+\lambda}})-w_{\frac{\lambda}{1+\lambda}} \|$. Then
\begin{eqnarray*}
d_k&\leq&\|\Pi_{i=k}^n  \D^{h_i} _\lambda (w_{\frac{\lambda}{1+\lambda}})- \D^{h_{k}} _\lambda (w_{\frac{\lambda}{1+\lambda}})\|  + \|\D^{h_{k}} _\lambda (w_{\frac{\lambda}{1+\lambda}})-w_{\frac{\lambda}{1+\lambda}} \|\\
&\leq&(1-\lambda h_{k})d_{k+1} +\|\D^{h_{k}} _\lambda (w_{\frac{\lambda}{1+\lambda}})-w_{\frac{\lambda}{1+\lambda}}\|.
\end{eqnarray*}
Now, for any $h$,
\begin{eqnarray*}
\|\D^{h} _\lambda (w_{\frac{\lambda}{1+\lambda}})-w_{\frac{\lambda}{1+\lambda}}\|&=&\|(1-h)(1-\lambda h)w_{\frac{\lambda}{1+\lambda}} +\lambda h \T(\frac{1-\lambda h}{\lambda}w_{\frac{\lambda}{1+\lambda}}) -w_{\frac{\lambda}{1+\lambda}}   \|\\
&\leq& \lambda h^2 \| w_{\frac{\lambda}{1+\lambda}} \| + \|\lambda h \T(\frac{1-\lambda h}{\lambda}w_{\frac{\lambda}{1+\lambda}}) -h(1+\lambda)w_{\frac{\lambda}{1+\lambda}} \|\\
&=& \lambda h^2 \| w_{\frac{\lambda}{1+\lambda}} \|+ \|\lambda h \T(\frac{1-\lambda h}{\lambda}w_{\frac{\lambda}{1+\lambda}}) -\lambda h\T(\frac{1}{\lambda}w_{\frac{\lambda}{1+\lambda}}) \|\\
&\leq& 2\lambda h^2 \| w_{\frac{\lambda}{1+\lambda}}\|\\
&\leq& 2\|\T(0)\|\lambda h^2.
\end{eqnarray*}

Hence 
\begin{eqnarray*}
d_k&\leq&(1-\lambda h_{k})d_{k+1}+ 2\|\T(0)\|\lambda h_k^2\\
&=&(1-\lambda h_{k})d_{k+1}+\lambda h_k (2\|\T(0)\| h_k)\\
&\leq&\max(d_{k+1}, 2\|\T(0)\|  h_k).
\end{eqnarray*}

Since $d_{n+1}=0$ we get $d_1\leq  2\|\T(0)\| \max_{1\leq i\leq n} h_i$ as claimed.
\end{proof}

In particular one gets 

 \begin{pro}\label{d2}
 For any $z \in Z$, and any  sequence $\{h_i\}$ with $h_i \leq h$ and $\sum_i h_i = + \infty$, 
 $$
 \| \Pi_{i=1}^\infty  \D^{h_i} _\lambda (z)-w_{\frac{\lambda}{1+\lambda}}\|\leq 2\|\T(0)\| h.
 $$
\end{pro}

\subsection{Asymptotic properties}$ $ \\
An easy consequence of Proposition \ref{dcv} is that for a given  $h$, $w_\lambda ^h$  has the same asymptotic behavior, as $\lambda$ tends to 0,  as $w_\lambda$.

\begin{pro}\label{d3}
\[
\|w_\lambda^h-w_\lambda\| \leq 2C\lambda.\]
\end{pro}
\begin{proof}
By Proposition \ref{basiclem},
\begin{eqnarray*}
\left\|w_\lambda^h-w_\lambda\right\|&=&  2 C  | 1 - (1 + \lambda - \lambda h |  \\
&\leq& 2C \lambda.
\end{eqnarray*}
\end{proof}

To generalize this property in order to apply it  to games with varying duration we need an additional assumption on the operator $\T$.

\begin{defi}
The operator $\T$ satisfies assumption (H) if there exists two nondecreasing functions $k:]0,1]\rightarrow\mathbb{R}^+$ and $\ell:[0,+\infty]\rightarrow\mathbb{R}^+$  with $k(\lambda)=o(\sqrt{\lambda})$ as $\lambda$ goes to 0 and
\[
\|\D^{1} _\lambda (z)-\D^{1} _\mu (z)\|\leq k(|\lambda-\mu|) \ell(\|z\|)
\]
for all $(\lambda,\mu)\in]0,1]^2$ and $z\in Z$.
\end{defi}

 \begin{pro}\label{d4}
If $\T$ satisfies (H) then for any $z \in Z$ and any  sequence $\{h_i\}$ with $\sum_i h_i = + \infty$, $  \| \Pi_{i=1}^\infty  \D^{h_i} _\lambda (z)-w_\lambda\|$ goes to 0 as $\lambda$ goes to 0.
\end{pro}
\begin{proof}
Since $\D^{h_i} _\lambda$ is $1-\lambda h$ contracting and $\sum_i h_i = + \infty$, $ \Pi_{i=1}^\infty \D^{h_i} _\lambda (z)$ is independent of $z$ and one may assume $z=w_\lambda$. \\Define $d_n=\| \Pi_{i=1}^n  \D^{h_i} _\lambda (w_\lambda)-w_\lambda\|$ hence $d_0=0$ and
\begin{eqnarray*}
d_n&\leq&\| \Pi_{i=1}^n  \D^{h_i} _\lambda (w_\lambda)- \D^{h_n} _\lambda (w_\lambda)\| +\|\D^{h_n} _\lambda (w_\lambda)-w_\lambda\|\\
&\leq & (1-\lambda h_n)d_{n-1}+\|\D^{h_n} _\lambda (w_\lambda)-w_\lambda\|.
\end{eqnarray*}
For any $h$, 
\begin{eqnarray*}
\|\D^{h} _\lambda (w_\lambda)-w_\lambda\|&\leq&\|(1-h)(1-\lambda h)w_\lambda+\lambda h \T (\frac{1-\lambda h}{\lambda}w_\lambda)-w_\lambda\|\\
&=&h(1+\lambda-\lambda h)\left\|\frac{\lambda}{1+\lambda-\lambda h} \T (\frac{1-\lambda h}{\lambda}w_\lambda)-w_\lambda\right\|\\
&=&h(1+\lambda-\lambda h)\left\|\D^{1} _\mu(w_\lambda)-\D^{1} _\lambda(w_\lambda)\right\| \text{ with $\mu=\frac{\lambda}{1+\lambda-\lambda h}$}\\
&\leq& h(1+\lambda-\lambda h)\ell(\|w_\lambda\|) k\left(\frac{\lambda^2(1-h)}{1+\lambda-\lambda h}\right) \text{ by (H)}\\
&\leq&2h\ell(\|T(0)\|)k(\lambda^2).
\end{eqnarray*}
Hence 
\begin{eqnarray*}
d_n&\leq& (1-\lambda h_n)d_{n-1}+ \lambda h_n \left[2\ell(\|T(0)\|)\frac{k(\lambda^2)}{\lambda}\right]\\
&\leq& \max \left(d_{n-1}, 2\ell(\|T(0)\|)\frac{k(\lambda^2)}{\lambda}\right)
\end{eqnarray*}
and $d_n\leq 2\ell(\|T(0)\|)\frac{k(\lambda^2)}{\lambda}$ for all $n$. The result follows since by assumption $k(\lambda^2)=o(\lambda)$.
\end{proof}

\subsection{Invariant properties}$ $ \\
We now consider another family of operators parametrized by $\alpha\in [0,1]$.\\
%
%
Define for $\alpha \in [0,1]$, $\widetilde{\T}_\alpha$ by
 \begin{equation}\label{tilde}
\widetilde{\T}_\alpha z = (1-\alpha)z +\T(\alpha  z ).
\end{equation}
 Thus $\widetilde{\T}_\alpha$   is non expansive, hence for $\lambda\in]0,1]$ one can consider the associated $\lambda$-discounted  fixed point  ${\widetilde{w}}^\alpha_\lambda$  defined by
\[
{\widetilde{w}}^\alpha_\lambda= \lambda \widetilde{\T}_\alpha(\frac{1-\lambda}{\lambda}\widetilde{w}^\alpha_\lambda).
\]
Note that for $\alpha = 1$, ${\widetilde{w}}^\alpha_\lambda = {w}_\lambda$.
\begin{pro}\label{d5}
$$
\widetilde{w}^\alpha_\lambda=w_\mu, \: \: \mbox{with }\ \mu= \frac{\lambda}{ \alpha + \lambda - \lambda \: \alpha }.
$$
\end{pro}

\begin{proof}
Direct computation gives 
\begin{eqnarray*}
{\widetilde{w}}^\alpha_\lambda& = &\lambda \widetilde{\T}_\alpha(\frac{1-\lambda}{\lambda}\widetilde{w}^\alpha_\lambda)\\
&=& \lambda[ (1-\alpha) \frac{(1-\lambda)}{\lambda}\widetilde{w}^\alpha_\lambda  + \T ( \alpha \frac{(1-\lambda)}{\lambda}\widetilde{w}^\alpha_\lambda)].
\end{eqnarray*}
Thus
\[
(\alpha +\lambda-\lambda \alpha){\widetilde{w}}^\alpha_\lambda=\lambda  \T ( \alpha \frac{(1-\lambda)}{\lambda}\widetilde{w}^\alpha_\lambda))
\]
which is 
$$
{\widetilde{w}}^\alpha_\lambda=\mu  \T (  \frac{1-\mu}{\mu}\widetilde{w}^\alpha_\lambda)
$$
  for $\mu= \frac{\lambda}{\alpha + \lambda - \lambda \: \alpha }$,
hence the result.
\end{proof}

\begin{cor} \label{inv}
For $\lambda\leq 1/2$, $\widetilde{w}^{\frac{\lambda}{1-\lambda}}_\lambda$ does not depend on $\lambda$ and equals $w_{1/2}$.
\end{cor}

\section{Finitely repeated exact games}
We consider the family of exact games $G^h = (h\, g, h \, Q) $ with $h\in [0,1]$.
\subsection{Approximation of the value} $ $ \\
The recursive equation  for the un-normalized  value $V_n^h $ of the $n$-stage game $G^h$ (of total length $nh$)   is given by:
$$
V_n^h (\omega) = \val [ h \:g( \omega; .) + P_h (.)  [\omega ]\circ  V_{n-1} ^h ]
$$
so that
$$
V_n^h  =  \T_h V_{n-1} ^h =  \T_h^n (0)
$$
Let  $f$ be  the solution of (\ref{M}) with $ \T$ satisfying (\ref{gT}).
Using the results of Section 4   in particular (\ref{itermu}) we obtain:
\begin{pro}\label{fin} $ $ \\
There exists a constant $L$ such that for all $n$ and $h \in [0,1]$
$$
\|V_n^h - f_{nh} (0) \| \leq  L h  \sqrt {n}.
$$
\end{pro}

\subsection{Vanishing step sizes} $ $ \\
The previous Proposition \ref{fin} shows that $v_n^h=\frac{ V_n^h}{N}$, which is the normalized value of the $n$-stage  game $G^h$ with length $N= nh$,  satisfies:
$$
\|v_n^h - \frac{ f_{N} (0)}{N} \| \leq  L\sqrt \frac{h}{  N}.
$$\\
In fact  Proposition \ref{kobavig} induces  a more precise result  for vanishing stage duration,  that we now describe.\\
Given $t>0$ and a  finite partition $H_t $ of $[0,t ], t_0= 0, ..., t_k = t $, induced by step sizes  $\{ h_i\}, 1\geq h_i >0, i = 1, ...,k, \sum_{i\leq j}  h_i = t_j$,   we define its mesh as $m (H_t) = \max_i h_i$.\\
 We consider the $k$ stage game where the duration of stage $i$ is $h_i$. Let $U( H_t)$ be its un-normalized value  (the normalized value is $u(H_t) = \frac {U( H_t)}{t}$). Thus $U( H_t) = \T^{h_1} \circ \cdots \circ  \T^{h_k} (0)$.
\begin{defi}   $\widehat V_t $ is the  limit value  on $[0, t ]$   if for any sequence of partitions   $\{H^n_t\}$ of $[0,t] $ with vanishing mesh,  the sequence of values   $\{U( H_t^n)\}$ of the corresponding games converges  to $\widehat  V_t $.
\end{defi}
 \begin{pro} \label{CN}  $ $ \\
There exists a constant $L'$ such that for any $H_t$ with $m(H_t) \leq h$,   the  un-normalized value $U ( H_t)$ satisfies 
$$
\|U ( H_t) - f_t (0)\|\leq L'  \sqrt{h  \, t}.
$$
 Thus the  limit value  $\widehat V_t $  exists and is given by
 $\widehat V_t = f_t(0)$.
 \end{pro}
 \begin{proof}
The inequality is obtained from equation (\ref{v2}) with $\sigma_k = t$ and $\tau_k \leq h\, t$.\\
The existence of $\widehat  V_t $ follows.
\end{proof}
The interpretation of these results is twofold:\\
first the value  of  the game with finite length is essentially independent of the duration of the stages when this duration is small enough,\\
second, this value is given by the solution of the associated differential equation \eqref{M}.\\

Note that  $\widehat V_t $  equals also the value of the continuous time game of length $t$ introduced in Neyman \cite{N2012}.

\subsection{Asymptotic analysis}$ $ \\
A further  consequence of Proposition \ref{CN} is that for any $t$ and any $k$-stage game associated to a finite partition $H_t$,    with normalized value $u(H_t)$,  one has: 

 \begin{pro} \label{ AS} $ $ \\
 There exists $L'$ such that for any $H_t$ 
 $$
 \|u(H_t) - \frac{f_t(0)}{t} \|\leq  \frac{L'}{\sqrt t}.
$$ 
 \end{pro}
In  particular the asymptotic behavior of the  (normalized) value of the game depends only on its  total length $t$  (and not  on the durations of the individual stages) up to a term $O ( \frac{1}{\sqrt t})$. Again the comparison quantity is given by the normalized solution of the associated differential equation \eqref{M}.\\

\section{Discounted  exact games}

\subsection{Values of discounted exact game}$ $ \\
%
%
%
We follow the definition of Neyman (eq. (3) p. 254 in \cite{N2013}) : the  (normalized) value $w_{\lambda}^h$ of the $\lambda$ discounted game $G^h$ is the unique solution of 
\[
w_{\lambda}^h (\omega) = \val_{X\times Y}[ h \lambda g(\omega,x,y) + (1- h {\lambda}) P_h (x,y)[\omega] \circ  w_{\lambda} ^h ].
\]
with $P_h=Id+hQ$.\\
In particular, for $h= 1$ one recovers $w_{\lambda} = w_{\lambda}^1$ (see 2.5) associated to $\T $ defined in (\ref{gT}).\\
The notation is consistent with the previous Section 5 since one has
\begin{pro}\label{id}
$w_{\lambda}^h$   corresponds to the solution of  (\ref{dh}).
 \end{pro}
 \begin{proof}
 $$
 w_\lambda^h= \lambda \T_h(\frac{1-\lambda h}{\lambda}w_\lambda^h) =\lambda [ \val_{X\times Y}[ h g(\omega,x,y) + P_h (x,y)[\omega] \circ  \frac{1-\lambda h}{\lambda}w_{\lambda} ^h ].
 $$
 Hence $w_\lambda^h=\lambda [ \val_{X\times Y}[ h \lambda  g(\omega,x,y) + (1- h {\lambda}) P_h (x,y)[\omega] \circ  w_{\lambda} ^h ].$
\end{proof}
A direct computation using $ P_h = Id + h\: Q$ gives

\begin{pro}\label{eqvlambdah}
$w_{\lambda}^h$  is the only solution of 
$$
\varphi (\omega) = \val_{X\times Y}[ g(\omega,x,y) +  \frac{(1- h {\lambda})}{\lambda}  Q  (x,y)[\omega] \circ   \varphi ].
$$
\end{pro}

We  now apply the results of Section 5

\begin{pro}\label{identif}
$$
w_{\lambda}^h=w_\mu, \: \: \mbox{with }\ \mu= \frac{\lambda}{1+ \lambda - \lambda \: h }.
$$
\end{pro}
\begin{proof}
Apply Proposition \ref{dcv}.
\end{proof}

\subsection{Vanishing duration}$ $ \\
 We now  recover the convergence property in \cite{N2013}.  


\begin{cor}\label{CW}$ $ \\
For a fixed $\lambda$,  $w_{\lambda}^h$ converges as $h$ goes to 0. The limit, denoted $\widehat w_{\lambda}$, equals   $w_\frac{\lambda}{1+\lambda}$, hence is the only solution of:
\begin{equation}\label{CD}
\varphi   =  \val [  g +  \frac{Q}{\lambda}  \circ   \varphi].
\end{equation}
Moreover, $\|w_{\lambda}^h-\widehat w_{\lambda}\|\leq C \lambda h$.
\end{cor}

\begin{proof}
For the convergence, apply Proposition \ref{d1}.\\
By definition $w_\frac{\lambda}{1+\lambda}$ satisfies  
\begin{equation*}
w_\frac{\lambda}{1+\lambda}  =  \val [  \frac{\lambda}{1+\lambda} g +  \frac{1}{1+\lambda} (Id+Q)  \circ   w_\frac{\lambda}{1+\lambda}].
\end{equation*}
that is
\begin{equation*}
w_\frac{\lambda}{1+\lambda}   =  \val [  g +  \frac{Q}{\lambda}  \circ   w_\frac{\lambda}{1+\lambda}].
\end{equation*}
\end{proof}
 More generally consider a sequence of stage durations $\{h_i\}$ with $h_i \leq h$ and $\sum_i h_i = + \infty$ inducing a partition $H$. The value of the associated $\lambda$-discounted game $W_\lambda^H$ is given by $   \Pi_{i=1}^{+ \infty}  \D^{h_i} _\lambda (0)$ hence satisfies
\begin{pro}\label{CW1}$ $ \\
$$
\|W_\lambda^H - \widehat w_{\lambda}\| \leq 2 \|\T (0)\| \: h \, .
$$
\end{pro}
\begin{proof}
Apply Proposition \ref{d2}.
\end{proof}

Once again $\widehat w_{\lambda}$ has to be interpreted as the $\lambda$-discounted value of the continuous time game, see  \cite{GH2}, \cite{N2012}, \cite{N2013}.\\
Note that  our game theoretic framework is very general, in particular there is no finiteness assumption on the actions or states. 

\subsection{Asymptotic behavior}$ $ \\
The  value $w_\lambda ^h$ of the $\lambda$ discounted game with  stage duration $h$   has the same asymptotic behavior, as $\lambda$ tends to 0,  as $w_\lambda$.

\begin{pro}
\[
\|w_\lambda^h-w_\lambda\| \leq 2C\lambda.\]
\end{pro}

\begin{proof}
Apply Proposition  \ref{d3}. 
\end{proof}

More generally one obtains
\begin{pro}
For any $\{h_i\}$ with $\sum_i h_i = + \infty$ inducing a partition $H$, $\|W_\lambda^H -  w_{\lambda}\| \leq C' \lambda$ where C' depends only on the game..
\end{pro}

\begin{proof}
Immediate consequence of Proposition \ref{d4} and its proof, since, by non expansiveness of the value operator, for any game with a payoff bounded by $C$  the associated Shapley operator $\T$ satisfies assumption (H) with $k(\lambda)= \lambda=o(\sqrt{\lambda})$ and $\ell(\|z\|)=C+\|z\|$.
\end{proof}
\subsection{Invariance properties}$ $ \\
Let $\T$ be the Shapley operator associated to the game $(g, Q)$. Then $\widetilde{\T}_\alpha$  defined by (\ref{tilde}) is the Shapley operator associated to $(g, \alpha \:Q)$ since 
\begin{eqnarray*}
\widetilde{\T}_\alpha (f)  &= &\val_{X\times Y}[ g(\omega,x,y) +  (Id + Q  (x,y)) [\omega] \circ  \alpha f ] + (1- \alpha)  f\\
&=& \val_{X\times Y}[ g(\omega,x,y) +  (Id + \alpha Q  (x,y)) [\omega] \circ  f ].
\end{eqnarray*}
This implies 
\begin{pro}
For any kernel $R$, the $\lambda$-discounted value of the game $G (g;  \frac{\lambda} {(1-  {\lambda})}R)$ is independent of $\lambda \leq 1/2$   and the only solution of 
$$
\varphi (\omega) = \val_{X\times Y}[ g(\omega,x,y) +  R  (x,y)[\omega] \circ   \varphi ].
$$
\end{pro}

\begin{proof}
Apply Corollary \ref{inv}.
\end{proof}
This shows a tradeoff between the size of the kernel and  the discount factor.  Taking into account Proposition \ref{identif}  one derives an invariance property of the value on the product space:  discount factor $ \times$  stage duration $ \times$  kernel:
$$
Val (\lambda, h , R) = Val ( \frac{\lambda}{1+ \lambda - \lambda \:h} , 1, R) = Val ( \lambda, 1, \frac {1- \lambda \:h }{1- \lambda}ÊR).
$$

Similar covariance properties were obtained in  \cite{N2012}  and \cite{N2013}.


\section{Discretization approach}

We consider now the game $\overline G ^h$ which corresponds to the discretization of the continuous time game. We will study two frameworks, like in the previous sections : either a fixed finite length or a fixed discount factor and we will analyse the behavior of the associated values as the stage duration $h$ goes to 0. 
\subsection{Finite length}$ $ \\
The un-normalized value $\overline{V}^h_n$ of the $n$-stage game with stage duration $h$ satisfies $\overline{V}^h_n= (\overline \T _h)^n (0)$. Similarly for varying stage duration, corresponding to a partition $H$,  one gets a recursive equation of the form $\overline{V}_H(t)= \Pi_i \overline \T_{h_i}(0)$.

\begin{lem} There exists $C_0$ such that
$$\|\T_h(f)-\overline{\T}_h(f)\|\leq C_0(1+\|f\|) h^2.$$
\end{lem}

\begin{proof}
By non expansiveness of the value operator, 
\begin{eqnarray*}
\|\T_h(f)-\overline{\T} _h(f)\|&\leq&\|h \:g(\cdot)-g^h(\cdot) \| + \|f\| \|P_h(\cdot)-\P^h(\cdot) \|\\
&=&h \:O(h)+ \|f\| h\: O(h)
\end{eqnarray*}
since $P_h=Id+h\:Q$ and  $\P^h=e^{hQ}=Id+h\:Q+ h \:O(h)$.
\end{proof}

\begin{pro}$ $ \\
There exists $C$ depending only of the game $G$ such that for any finite sequence $(h_i)_{i\leq n}$ in $[0,h]$ with sum $t$ and corresponding partition $H$:
\[
\| \overline V _H (t)-\widehat{V}_t  \|\leq C(\sqrt{ht}+ht+ht^2). 
\] 
In particular for a given $t$,  $\overline{V}_H (t)$ tends to $\widehat{V}_t$ as $h$ goes to 0. 
\end{pro}

\begin{proof}
The value of any game with total length less than $t$ is bounded by  some $C_1$t, independently of $h$. Hence non expansiveness of the operators as well as the previous Lemma 8.1 gives
\begin{eqnarray*}
\|\overline V_H (t)-\Pi_i \T_{h_i}(0)\| &=&\|\Pi_i \overline\T_{h_i}(0)-\Pi_i \T_{h_i}(0)\| \\
&\leq&\left\|\overline\T_{h_1}\left(\underset{i\geq2}{\Pi} \overline\T_{h_i}(0)\right)-\overline\T_{h_1}\left(\underset{i\geq2}{\Pi} \T_{h_i}(0)\right)\right\| \\
& & +\left \| \overline\T_{h_1}\left (\underset{i\geq2}{\Pi}\T_{h_i}(0)\right)-\T_{h_1}\left (\underset{i\geq2}{\Pi} \T_{h_i}(0)\right)  \right \| \\
&\leq&\left\|\underset{i\geq2}{\Pi} \overline\T_{h_i}(0)-\underset{i\geq2}{\Pi} \T_{h_i}(0)\right\|+C_0h_1^2\left(1+C_1\sum_{i=2}^n h_i\right).
\end{eqnarray*}
Without loss of generality $C_1\geq1$ hence by summation,
\begin{eqnarray*}
\|\overline V_H(t)-\Pi_i \T_{h_i}(0)\| &\leq&C_0C_1(1+t)\sum_{i=1}^n h_i^2 \\
&\leq&C\: h(1+t)\sum_{i=1}^n h_i\\
&=&C\:h\:t(1+t)
\end{eqnarray*}
for $C=C_0C_1$. Then Proposition \ref{CN} yields the result.
\end{proof}

\begin{rem}
For a given  $h$, the right hand term is quadratic in $t$, hence we do not link the asymptotic behavior of the normalized quantity $\overline{v}^h_n=\frac{\overline{V}^h_n}{nh}$ and of $\widehat{v}_t=\frac{\widehat{V}_t}{t}$. However if $n$ is a function of $h$ converging slowly enough to infinity,  the previous proposition can be used. For example for $n(h)=\frac{1}{h\sqrt{h}}$ (so that $t(h)=\frac{1}{\sqrt{h}}$), one has
\[
\| \overline v^h_{n(h)}-\widehat{v}_{t(h)} \|=O(\sqrt{h}).
\] 

 \end{rem}
\subsection{Discounted case}$ $ \\
We  consider  uniform stage duration $h$. The normalized value $\overline{w}^h_k$ of the discretization with mesh $h$  of the $\lambda$-discounted continuous game satisfies the fixed point  equation

\[
\overline{w}^h_\lambda(\omega)=\val_{X\times Y}\left [\int_0^h \lambda e^{-\lambda t} g(\omega_t,x,y) + e^{-\lambda h} \P^h (x,y)[\omega] \circ  \overline{w}_{\lambda} ^h \right].
\]

\begin{pro}$ $ \\
For a given  $\lambda$, $\overline{w}^h_\lambda$ tends to $\widehat{w}_\lambda$ as $h$ goes to 0.
\end{pro}

\begin{proof}
The equations for $\overline{w}^h_\lambda$ and $w^h_\lambda$, as well as the non expansiveness of the value operator, give:
\begin{eqnarray*}
\|\overline{w}^h_\lambda-w^h_\lambda\|&\leq& \lambda h\left\| g(\cdot)-\frac{1}{h}\int_0^h e^{-\lambda t}g_t(\cdot) \right\|+e^{-\lambda h}\|\overline{w}^h_\lambda-w^h_\lambda\| +  \|w^h_\lambda \| \|P_h-\P^h \| \\
&&+(1-\lambda h-e^{-\lambda h})\|w^h_\lambda \|\\
&\leq&\lambda O(h^2)+e^{-\lambda h}\|\overline{w}^h_\lambda-w^h_\lambda\| + O(h^2)+ \lambda^2 O(h^2)
\end{eqnarray*}
hence for a fixed $\lambda$, $(1-e^{-\lambda h})\|\overline{w}^h_\lambda-w^h_\lambda\|=O(h^2)$ and the result follows from Corollary \ref{CW}.
\end{proof}
Similar properties were obtained in \cite{N2013}.

\bigskip

For an alternative approach to the limit behavior of the discretization of the continuous model,  relying on viscosity solution tools and extending to various information structures on the state, see \cite{S15}.

\section{Extensions and concluding comments} 


\subsection{Stochastic games: no signals on the state}$ $ \\
Consider  a finite stochastic game  where the players know only the initial distribution $m \in \Delta(\Omega)$ and the actions at each stage.\\
The basic equation  for the exact game with duration $h$ is then
$$
\hat \T^h f (m) = \val [ h g(m; x,y) + \sum_{ij} x^i y^j f ( m\ast P_h(i, j) ) ].
$$
with $ [m\ast P_h(i, j)]( \omega) = \sum_z  m (z)  P_h(i, j)[z] ( \omega)$ being the image of the probability $m$ by the kernel $P_h(i, j)$.\\
The equation 
$$
 \hat \T^h = h  \hat \T  + (1-h) Id
 $$
 does not hold anymore and $ \T - Id $ has to be replaced by  $ \lim_{ h \rightarrow 0} \frac { \hat \T ^h   - Id}{h}$ in (\ref{M}). The study of such games with varying duration thus seems more involved.
 
\subsection{Link with games with uncertain duration}$ $ \\
Notice that $\T_h = (1-h) Id + h \T$ is a particular case of an operator of the form $ \sum_i \alpha_i \T^i, \alpha_i \geq 0, \sum \alpha_i = 1,$ which corresponds to some generalized iterate \cite{NNE,NS10} of $\T$. 
Hence all the values computed in sections 4.3, 5.3 and so on, can also be seen as the value of some games with uncertain duration. See \cite{V1} for specific remarks in the particular case of $V_n^h$.


\subsection{Oscillations} $ $ \\
Several examples of stochastic games (either with a finite set of states and compact sets of actions \cite{V4}, or compact set of states and finite set of actions \cite{Zi1}) were recently constructed for which the values $v_n$ and $v_\lambda$ do not converge. Hence the values of the corresponding games with vanishing duration (and  thus their limit as continuous time games) do not converge when $t$ goes to infinity or $\lambda$ to 0.

\subsection{Comparison to the literature} $ $ \\
The approach here is different from the one of Neyman \cite{N2013} : the proofs are based on properties of operators and not on strategies. For example  \cite{N2013} shows that playing optimally in (\ref{CD})   will imply  Corollary \ref{CW}.\\
By comparison our tools  consider only the values and apply to any non expansive map $\T$.

\subsection{Main results} $ $ \\
The main results can be summarized  in two parts:\\
- for a given finite length (or  discounted evaluation) the value of the game with vanishing stage duration converges  thus defining a limit value for the  associated continuous time game. Moreover the limit is described explicitly.\\
- as the length goes to $\infty$  or  the discount factor goes to 0, the impact of the stage duration goes to 0  and the asymptotic behavior of the normalized value function  is independent of the discretization.

\end{document}